\documentclass[10pt]{article}
\usepackage{amsthm, amsmath, amssymb, amsbsy, amsfonts}

\newtheorem{thm}{Theorem}
\newtheorem{prop}{Proposition}
\newtheorem{lemma}{Lemma}
\def\eps{\varepsilon}

\begin{document}


\title{Homogenization of non-autonomous operators of convolution type
in periodic media.}

\author{A. Piatnitski$^{1,2}$ \  and \ E. Zhizhina$^{2,1}$\\[4mm]
{$^1$\ \small The Arctic University of Norway, UiT, campus Narvik, Norway}\\
 {\ \,  $^2$  \small Institute for Information Transmission Problems of RAS, Moscow, Russia}\\
{\small emails:} \ {\small\tt apiatnitski@gmail.com},\ {\small\tt ejj@iitp.ru}}
\date{}





\maketitle

\noindent
{\bf Keywords:}\ {convolution type operators, periodic homogenization, homogenization of non-autonomous
equations.}

\medskip\noindent
{\bf MSC}:\ {35B27, 45E10, 60J76, 45M15}.
\begin{abstract}
The paper deals with periodic homogenization problem for a  para\-bo\-lic equation whose elliptic part is a convolution type
operator with rapidly oscillating coefficients.  It is assumed that the coefficients  are rapidly oscillating periodic functions  both
in spatial and temporal variables and that the scaling is diffusive that is the scaling factor of the temporal variable
is equal to the square of the scaling factor of the spatial variable.
Under the assumption that the convolution kernel has a finite second moment and that the operator is symmetric in spatial
variables we show that the studied equation admits homogenization and prove that the limit operator is a second order
differential parabolic operator with constant coefficients.
\end{abstract}

\section{Introduction}
\label{s_intro}
The goal of this work is to study the limit behaviour of solutions of a Cauchy problem for nonlocal equations
of the form
\begin{equation}\label{ori_eqn}
\begin{array}{c}
\displaystyle
\partial_t u^\eps(x,t)=\mathcal{A}^\eps(t) u^\eps(x,t)\qquad \hbox{in }\mathbb R^d\times[0,T],\\[2mm]
u^\eps(x,0)=u_0(x),\quad u_0\in L^2(\mathbb R^d),
\end{array}
\end{equation}
where $\eps>0$ is a small positive parameter that characterize the microscopic length scale of the medium,
and $\mathcal{A}^\eps(t)$ are  convolution type nonlocal operators defined by
\begin{equation}\label{ori_oper}
\big(\mathcal{A}^\eps(t) v\big)(x,t)=\frac1{\eps^{d+2}}\int_{\mathbb R^d}a\Big(\frac{x-y}\eps\Big)
\mu\Big(\frac{x}\eps,\frac y\eps, \
\frac t{\eps^2}\Big)\big(v(y)-v(x)\big)dy,
\end{equation}
where the kernel $a(\cdot)$ is a non-negative, intergable  function that has finite second moments,
and the function $\mu$ is strictly positive and bounded. We suppose furthermore the following symmetry conditions:
$a(-z)=a(z)$ for all $z\in
\mathbb R^d$, and $\mu(\xi,\eta,s)=\mu(\eta,\xi,s)$ for all  $\xi,\,\eta\in \mathbb R^d$ and $s\in\mathbb R^+$.
Also, it is assumed that the finction $\mu$ is periodic in all variables. It means that the medium has a periodic microstructure,
and its evolution in time is also periodic.  The detailed conditions on the operator $\mathcal{A}^\eps$ are formulated in
the beginning of the next section.

Convolution type operators with an integrable kernel are used in population dynamics, for instance to describe the evolution
of a population density, in some models of porous media, in financial mathematics, and in some other fields.
The presence of function $\mu$ is due to inhomogeneity of the medium, this function represents the local characteristics
of the medium.

In the paper we consider a model periodic problem and assume that the environment has a periodic microstructure
and that the evolution of its characteristics in time is also periodic.

%

Differential operators with rapidly oscillating coefficients and the corresponding homogenization problems have been actively
 studied since 70s of the last century. There is a vast literature devoted to this topic. We refer here just two monographs,
 \cite{BLP} and \cite{JKO}.


Recent time various homogenization problems for nonlocal operators attract the attention of many mathematicians.
The operators in which a nonlocal operator is a perturbation of a local elliptic operator were considered in
 \cite{San16} and \cite{Schw10}.  The works \cite{Ari09}, \cite{KPZ19} and \cite{RhVa09} deal with homogenization of L\'evy type
 operators with non-integrable kernel, the limit operator being a L\'evy type operator.


In the case of equations with time independent coefficients homogenization results for problem  \eqref{ori_eqn} were
obtained in our previous works \cite{PiaZhi17}, \cite{PZ19}. It was shown that the effective equation is a second order
differential parabolic equation with constant coefficients. In the non-symmetric case homogenization takes place in moving
coordinates, see \cite{PZ19}.


To our best knowledge homogenization problems for nonlocal convolution type  operators with time-dependent coefficients
have not been studied in the existing literature.


The goal of this work is to show that problem  \eqref{ori_eqn}, \eqref{ori_oper} admits homogenization, to describe the homogenized
model and to justify the convergence. We will show that the homogenized equation is a second order parabolic equation with constant coefficients, 
the effective diffusion matrix is defined in terms of solutions to auxiliary periodic problems on $(d+1)$-dimensional torus.


Our approach is based on constructing the three main  terms of the asymptotic expansion of a solution to problem \eqref{ori_eqn},
the initial condition being sufficiently smooth. All the terms except for the first one contain correctors which are introduced
as solutions of auxiliary cell problems. These problems are non-standard because periodic boundary conditions are imposed
not only in spatial variables but also in time.  
The mentioned  non-standard cell problems 
form the main novelty of this work.


\section{Main assumptions and the result}
\label{s_setup}

We assume that the following conditions are fulfilled:
\begin{itemize}
  \item [$\bf C_1$]  Non-negativity and intergability of the convolution kernel:
  $$
  a(z)\geqslant 0\quad\hbox{for all }z\in\mathbb R^d, \qquad \int_{\mathbb R^d}a(z)dz=1.
  $$
  \item  [$\bf C_2$]  Finiteness of the second moments:
  $$
   \int_{\mathbb R^d}|z|^2a(z)dz<\infty.
  $$
   \item [$\bf C_3$]  Uniform ellipticity: there exist  $\mu^->0$ and $\mu^+>0$ such that
  $$
  \mu^-\leqslant\mu(\xi,\eta,s)\leqslant\mu^+\ \ \ \hbox{ for all }\xi,\,\eta, \,s.
  $$
   \item [$\bf C_4$]  Periodicity: the function $\mu=\mu(\xi,\eta,s)$ is periodic in all variables $\xi$, $\eta$ and $s$.
   Without loss of generality we assume that the period equals one for each coordinate direction.
  \item [$\bf C_5$] Symmetry:
  $$
  a(-z)=a(z)\ \ \hbox{for all }z\in\mathbb R^d,\qquad \mu(z,\xi,s)=\mu(\xi,z,s) \ \ \hbox{for all }z,\,\xi,\, s.
  $$
\end{itemize}

Consider a Cauchy problem of the form
\begin{equation}\label{gener_eqn}
\begin{array}{c}
\displaystyle
\partial_t v^\eps(x,t)=\mathcal{A}^\eps(t) v^\eps(x,t)+f(x,t)\qquad \hbox{in }\mathbb R^d\times[0,T],\\[2mm]
v^\eps(x,0)=v_0(x),
\end{array}
\end{equation}
where $f\in L^2(\mathbb R^d\times(0,T))$, and $v_0\in L^2(\mathbb R^d)$. We recall that the operator $\mathcal{A}^\eps(t)$  has been defined in \eqref{ori_oper}.

According to the Schur lemma for integral operators, see \cite{Gra}, for any $t\in\mathbb R$ the operator $\mathcal{A}^\eps(t)$ is bounded in $L^2(\mathbb R^d)$,
moreover,  $\|\mathcal{A}^\eps(t)\|_{\mathcal{L}(L^2(\mathbb R^d),L^2(\mathbb R^d))}\leqslant2\eps^{-2}\mu^+ $ for
any $t\in\mathbb R$. By the standard arguments of the theory of parabolic equations this implies that for each $\eps>0$ problem
\eqref{gener_eqn}  has a unique solution in the space $L^\infty(0,T; L^2(\mathbb R^d))$.


For arbitrary functions $v$ and $w$ from $L^2(\mathbb R^d)$ denote by $(v,w)$ their inner product in $L^2(\mathbb R^d)$.
Also for the sake of brevity we use the notation $\mu^\eps(x,y,t)=\mu\big(\frac x\eps,\frac y\eps,
\frac t{\eps^2}\big)$ and $a^\eps(z)=a\big(\frac z\eps\big)$.

The main result of this work reads
\begin{thm}\label{thm_mainsetup}
There exists a positive definite constant matrix $a^{\rm eff}$ such that for any $u_0\in L^2(\mathbb R^d)$
a solution of problem  \eqref{ori_eqn}--\eqref{ori_oper} converges, as $\eps\to0$, to a solution of the Cauchy problem
$$
\partial_t u=\mathrm{div}\big(a^{\rm eff}\nabla u\big),\qquad u(x,0)=u_0(x).
$$
For the definition of matrix $a^{\rm eff}$ see Section \ref{s_cellpbm}.
\end{thm}

\section{Auxiliary periodic cell problems.}
\label{s_cellpbm}

In what follows 
  we  identify periodic functions in $\mathbb R^d$ with functions defined on the standard $d$-dimensional torus; the functions periodic
  both in spatial variables and in time are identified with those defined on $\mathbb T^{d+1}$.


Consider the following equation:
\begin{equation}\label{aux1}
  \partial_s \beta(\xi,s)-\mathcal{A}(s)\beta(\xi,s)= \theta(\xi,s)
\end{equation}
with  $\theta\in L^2(\mathbb T^{d+1})$ and
\begin{equation}\label{def_cal_A}
\mathcal{A}(s)\beta(\xi,s):=\int_{\mathbb R^d}a(\xi-\eta)\mu(\xi,\eta,s)(\beta(\eta,s)-\beta(\xi,s))d\eta.
\end{equation}
We consider the operators
$(\partial_s-\mathcal{A}(s))$ and its adjoint $(-\partial_s-\mathcal{A}(s))$ in $L^2(T^{d+1})$. Both operators are equipped
with a domain $H^1(\mathbb T^1; L^2(\mathbb T^d))$.
The compatibility condition for equation  \eqref{aux1} is given by the following statement:
\begin{prop}
\label{p_auxi} Let conditions $\bf C_1$-$\bf C_5$ hold.
Then the kernels of the operators $(\partial_s-\mathcal{A}(s))$ and $(-\partial_s-\mathcal{A}(s))$ consist of constants only.
Equation \eqref{aux1} has a solution $\beta\in L^2(\mathbb T^{d+1})$ if and only if
\begin{equation}\label{nes_suf}
\int_0^1\int_{\mathbb T^d}\theta(\xi,s)d\xi ds=0.
\end{equation}
The solution is unique up to an additive constant. If the average of $\beta$ over $\mathbb T^{d+1}$ vanishes, then 
the following estimate holds:
\begin{equation}\label{aux2}
  \|\beta\|_{L^\infty(0,1;L^2(\mathbb T^d))}\leqslant C
  \|\theta\|_{L^2(\mathbb T^{d+1})}
\end{equation}
with a constant $C>0$.
\end{prop}
\begin{proof}
%
 First we show that the kernel of the operator $(\partial_s-\mathcal{A}(s))$ defined on $L^2(\mathbb T^{d+1})$
 consists of constants only. Consider a periodic solution $\rho=\rho(\xi,s)$ of the equation $\partial_s\rho-\mathcal{A}(s)\rho=0$.
 Multiplying this equation by $\rho$ and integrating the resulting relation over $\mathbb T^{d+1}$ we arrive at the relation
 $$
 \frac12
 \int_0^1ds\int_{\mathbb T^d}\int_{\mathbb R^d}\mu(\xi,\eta,s)a(\xi-\eta)\big(\rho(\eta,s)-\rho(\xi,s)\big)^2d\xi d\eta=0.
 $$
Since $a(\cdot)\geqslant 0$ and $\mu(\cdot)\geqslant \mu_->0$,  this relation holds if and only if
 $$
\frac12 \int_0^1ds\int_{\mathbb T^d}\int_{\mathbb R^d}a(\xi-\eta)\big(\rho(\eta,s)-\rho(\xi,s)\big)^2d\xi d\eta=0.
 $$
 Therefore, for almost all $s\in (0,1)$ we  have
 $$
\frac12 \int_{\mathbb T^d}\int_{\mathbb R^d}a(\xi-\eta)\big(\rho(\eta,s)-\rho(\xi,s)\big)^2d\xi d\eta=0.
 $$
The expression on the left-hand side here is the quadratic form of the operator
$$
\rho(\xi,s)-\int_{\mathbb R^d}a(\xi-\eta)\rho(\eta,s)d\eta=:\mathcal{B}\rho(\xi,s),
$$
where $s$ is a parameter. Since this operator is self-adjoint and the operator $\rho\mapsto \int_{\mathbb R^d}a(\xi-\eta)\rho(\eta)d\eta$
is compact in $L^2(\mathbb T^d)$, then zero is an eigenvalue of $\mathcal{B}$, and the dimension of the kernel
of quadratic form
$$
(\mathcal{B}\rho,\rho)_{L^2(\mathbb T^d)}=\frac12\int_{\mathbb T^d}\int_{\mathbb R^d}a(\xi-\eta)\big(\rho(\eta,s)-\rho(\xi,s)\big)^2d\xi d\eta
$$
coincides with the multiplicity of this eigenvalue.   As was shown in \cite{PZ19},  zero is a simple eigenvalue
of $\mathcal{B}$. Therefore, the kernel of $(\partial_s-\mathcal{A}(s))$ consists of constants only.
 By the same reason the kernel of  the adjoint operator $(-\partial_s-\mathcal{A}(s))$ contains
only constants.

The solvability condition of equation \eqref{aux1} is a non-trivial issue. We show that this problem can be reduced to
solvability problem for some Fredholm operator.
In order to prove the second statement of Proposition \ref{p_auxi} and  justify compatibility condition \eqref{nes_suf} we
consider an auxiliary Cauchy problem
\begin{equation}\label{cauchy_super}
\partial_s\gamma(\xi,s)=\mathcal{A}(s)\gamma(\xi,s)+\theta(\xi,s), \quad\gamma(\xi,0)=\nu(\xi), \quad
(\xi,s)\in\mathbb T^d\times(0,1).
\end{equation}
Clearly, equation  \eqref{aux1}  is solvable if and only in for some $\nu\in L^2(\mathbb T^d)$ we have $\gamma(\cdot,1)=\nu$.
In order to solve the equation $\gamma(\cdot,1)=\nu$ we introduce two more Cauchy problems.
The first one reads 
\begin{equation}\label{cau_pbm_psi}
\partial_s\psi-\mathcal{A}(s)\psi=0,\quad \psi(\xi,0)=\nu(\xi), \qquad \nu\in L^2(\mathbb T^d).
\end{equation}
We denote by $\mathcal{S}$ the operator in $L^2(\mathbb T^d)$ that maps the initial condition $\nu(\cdot)$ to $\psi(\cdot,1)$.
The second Cauchy problem reads
\begin{equation}\label{cauchy_fred}
\partial_s\phi-\mathcal{A}(s)\phi=\theta,\quad \phi(\xi,0)=0, \qquad \theta\in L^2(\mathbb T^d\times(0,1)),
\end{equation}
 its solution evaluated at $s=1$ is denoted by $\hat\phi$,  $\hat\phi(\xi)=\phi(\xi,1)$.
Then the relation $\gamma(\cdot,1)=\nu$ is equivalent to the equation $\mathcal{S}\nu+\hat\phi=\nu$
that can be rewritten as
$$
(\mathcal{S}-I)\nu+\hat\phi=0.
$$

Evidently, under our standing assumptions problems \eqref{cauchy_super}--\eqref{cauchy_fred} have a unique solution.
Moreover, the operator $\mathcal{S}$ is bounded in $L^2(\mathbb T^d)$, and the kernel of the operator
$\mathcal{S}-I$ consists of constants only.  To define the adjoint operator $\mathcal{S}^*$ we consider
the Cauchy problem
$$
-\partial_s\vartheta-\mathcal{A}(s)\vartheta=0,\quad \vartheta(\xi,1)=\nu, \qquad (\xi,s)\in \mathbb T^d\times(0,1).
$$
Then $\mathcal{S}^*\nu(\xi)= \vartheta(\xi,0)$. Exploiting the same arguments as above we conclude
that the kernel of $\mathcal{S}^*-I$ also consists of constants only.

We are going to show that the operator $\mathcal{S}-I$ can be represented as the sum of a compact and an invertible operators
in $L^2(\mathbb T^d)$. To this end we introduce the notation
$$
\mathcal{A}_0(s)u(\xi)=\int_{\mathbb R^d}a(\xi-\eta)\mu(\xi,\eta,s)u(\eta)\,d\eta,\quad
G(\xi,s)=\int_{\mathbb R^d}a(\xi-\eta)\mu(\xi,\eta,s)\,d\eta.
$$
Obserbve  that $G$ is a periodic in  $\xi$ and  $s$ function that satisfies the estimates
\begin{equation}\label{bou_g}
0<\mu^-\leq G(\xi,s)\leq \mu^+.
\end{equation}
In problem \eqref{cau_pbm_psi} we make the following change  of unknown function:
\begin{equation}\label{chan_vary}
\psi(\xi,s)=\exp\Big(-\int_0^sG(\xi,\tau)\,d\tau\Big)\Psi(\xi,s).
\end{equation}

Then the function $\Psi(\xi,s)$ is a solution to the Cauchy problem
\begin{equation}\label{eqG}
\partial_s\Psi-\mathcal{A}_G(s)\Psi=0, \qquad \Psi(\xi,0)=\nu(\xi),
\end{equation}
with
\begin{equation}\label{aux_cau_ppsi}
\mathcal{A}_G(s)\Psi(\xi,s)=\int_{\mathbb R^d}a(\xi-\eta)\mu(\xi,\eta,s)\exp\Big(\int_0^s\big(G(\xi,\tau)-G(\eta,\tau)\big)\,d\tau\Big)
\Psi(\eta,s)\,d\eta.
\end{equation}
The operator that maps the initial condition $\nu(\xi)$ to the solution $\Psi(\xi,s)$, $0\leqslant s\leqslant 1$, is denoted by $\mathcal M$.
Since the family of operators $\{\mathcal{A}_G(s)\}$ is uniformly in $s$
bounded in $L^2(\mathbb T^d)$,  problem \eqref{eqG}
has a unique solution for each $\nu\in L^2(\mathbb T^d)$,
and the following inequality holds:
$$
\|\Psi\|_{L^\infty(0,1;L^2(T^d))}\leq C\|\nu\|_{L^2(\mathbb T^d)}.
$$
Therefore, $\mathcal{M}$ 
is a bounded linear operator from $L^2(\mathbb T^d)$  to
$L^\infty(0,1;L^2(\mathbb T^d))$.
From \eqref{eqG} and  \eqref{aux_cau_ppsi} we obtain
\begin{equation}\label{chewing}
\Psi(\xi,1)=\nu(\xi)+\int_{0}^{1}\int_{\mathbb R^d}a(\xi-\eta)\mu(\xi,\eta,s)\exp\Big(\int_0^s\big(G(\xi,\tau)-G(\eta,\tau)\big)\,d\tau\Big)
\Psi(\eta,s)\,d\eta ds
\end{equation}
The integral operator on the right-hand side here defines a compact linear operator from $L^2(\mathbb T^d\times[0,1])$ to
$L^2(\mathbb T^d)$.
Indeed, the kernel of the integral operator in \eqref{chewing} admits an estimate
$$
a(\xi-\eta)\mu(\xi,\eta,s)\exp\Big(\int_0^s\big(G(\xi,\tau)-G(\eta,\tau)\big)\,d\tau\Big)\leqslant \mu^+e^{\mu^+}a(\xi-\eta).
$$
According to \cite[Proposition 6]{PiaZhi17}, for  each $s\in(0,1)$  the norm of the operator
$$
\Phi(\xi)\mapsto \int_{\mathbb R^d}a(\xi-\eta)\mu(\xi,\eta,s)\exp\Big(\int_0^s\big(G(\xi,\tau)-G(\eta,\tau)\big)\,d\tau\Big)
\Phi(\eta)\,d\eta
$$
in $L^2(\mathbb T^d)$ does not exceed $\|a\|_{L^1(\mathbb R^d)}\mu^+e^{\mu^+}=\mu^+e^{\mu^+}$. Consequently,
the norm of the operator
$$
\Psi\mapsto
\int_{0}^{1}\int_{\mathbb R^d}a(\xi-\eta)\mu(\xi,\eta,s)\exp\Big(\int_0^s\big(G(\xi,\tau)-G(\eta,\tau)\big)\,d\tau\Big)
\Psi(\eta,s)\,d\eta ds
$$
acting from $L^2(T^{d+1})$ to $L^2(\mathbb T^d)$ admits the same upper bound.  Approximating $a(z)$ in $L^1(\mathbb R^d)$
by functions from $L^2(\mathbb R^d)$ and making use of the same arguments as in the proof of  \cite[Proposition 6]{PiaZhi17}
we obtain the desired compactness.

Letting
$$
\mathcal{K}\nu(\xi)=\int_{0}^{1}\int_{\mathbb R^d}a(\xi-\eta)\mu(\xi,\eta,s)\exp\Big(\int_0^s\big(G(\xi,\tau)-G(\eta,\tau)\big)\,d\tau\Big)
\mathcal{M}\nu(\eta,s)\,d\eta ds
$$
one can rewrite equation \eqref{chewing} as
$$
\Psi(\cdot,1)=\nu+\mathcal{K}\nu,
$$
where $\mathcal{K}$ is a compact operator in $L^2(\mathbb T^d)$ because it is a composition of  a bounded operator from
$L^2(\mathbb T^d)$ to $L^2(\mathbb T^{d+1})$ and a compact operator from $L^2(\mathbb T^{d+1})$ to $L^2(\mathbb T^d)$.
In view of \eqref{chan_vary} this yields
$$
\mathcal{S}\nu=\psi(\cdot,1)=\exp\Big(-\int_0^1G(\xi,s)\,ds\Big)\nu+\mathcal{K}_1\nu,
$$
with a compact operator $\mathcal{K}_1$. The equation $\mathcal{S}\nu-\nu=
-\hat\phi$ can now be rewritten as
\begin{equation}\label{fredh}
\Big[\exp\Big(-\int_0^1G(\xi,s)\,ds\Big)-1\Big]\nu+\mathcal{K}_1\nu=-\hat\phi.
\end{equation}
Due to \eqref{bou_g} the multiplication operator $\nu\,\to\, \Big[\exp\Big(-\int_0^1G(\xi,s)\,ds\Big)-1\Big]\nu$
is invertible.  Since the kernel of the adjoint operator $(\mathcal{S}^*-I)$ consists of constants only,
by the Fredholm theorem (see e.g. \cite{RS}) equation \eqref{fredh} is solvable if and only if
$$
\int_{\mathbb T^d}\hat\phi(\xi)d\xi=0.
$$
Integrating the equation in \eqref{cauchy_fred}  over $\mathbb T^d\times(0,1)$
and using the relation
$$
\big( \mathcal{A}(s) \phi, 1 \big)_{L_2({\mathbb T^d})} =0 \quad \mbox{ for any } \; s,
$$
we obtain
$$
\int_{\mathbb T^{d+1}}\theta(\xi,s)d\xi ds=\int_{\mathbb T^d}\hat\phi(\xi)d\xi.
$$
Therefore,  condition \eqref{nes_suf} is necessary and sufficient for solvability of equation \eqref{aux1}.

\medskip

In order to justify \eqref{aux2} we denote
$$
\overline{\beta}(s)=\int_{\mathbb T^d}\beta(\xi,s)\,d\xi, \quad \overline{\theta}(s)=\int_{\mathbb T^d}\theta(\xi,s)\,d\xi,
\quad\mathop{\beta}\limits^{\scriptscriptstyle \circ}(\xi,s)=\beta(\xi,s)-\overline{\beta}(s),\quad
\mathop{\theta}\limits^{\scriptscriptstyle \circ}(\xi,s)=\theta(\xi,s)-\overline{\theta}(s)
$$
and observe that both $\mathop{\theta}\limits^{\scriptscriptstyle \circ}$ and $\overline{\theta}$ have zero average over
$\mathbb T^{d+1}$ if the average of $\theta$ vanishes. Therefore, $\overline{\beta}(s)$ and
$\mathop{\beta}\limits^{\scriptscriptstyle \circ}(\xi,s)$ are solutions of the equations
\begin{equation}\label{equ_parts}
  \partial_s \beta-\mathcal{A}(s)\beta=\overline{\theta}(s)\qquad\hbox{and}\qquad
  \partial_s \beta-\mathcal{A}(s)\beta=\mathop{\theta}\limits^{\scriptscriptstyle \circ}(\xi,s),
\end{equation}
respectively.
Since $\overline{\theta}$ does not depend on $\xi$ due to the definition of $\mathcal{A}(s)$, see \eqref{def_cal_A}, the first equation is reduced to $\partial_s\overline{\beta}(s)
=\overline{\theta}(s)$ and we trivially have $\|\overline{\beta}\|_{L^\infty(\mathbb T^{d+1})}\leqslant
\|\overline{\theta}\|_{L^2(\mathbb T^{d+1})}$ if the mean value of $\overline{\beta}$ vanishes.
Multiplying the second  equation in \eqref{equ_parts} by $\mathop{\beta}\limits^{\scriptscriptstyle \circ}$
and integrating the resulting relation over $\mathbb T^{d+1}$  yields
\begin{equation}\label{start_ineq}
\frac12\int_{\mathbb T^{d+1}}\int_{\mathbb R^d}a(\xi-\eta)\mu(\xi,\eta,s)\big(
\mathop{\beta}\limits^{\scriptscriptstyle \circ}(\eta,s)-\mathop{\beta}\limits^{\scriptscriptstyle \circ}(\xi,s))\big)^2\,d\eta d\xi ds
=\int_{\mathbb T^{d+1}} \mathop{\beta}\limits^{\scriptscriptstyle \circ}(\xi,s)
\mathop{\theta}\limits^{\scriptscriptstyle \circ}(\xi,s)d\xi ds.
\end{equation}
As was shown above,  zero is a simple eigenvalue of the operator
$$\mathcal{A}_-v(\xi):=
\mu_-\int_{\mathbb R^d}a(\xi-\eta)\big(v(\xi)-v(\eta)\big) d \eta
$$ in $L^2(\mathbb T^d)$.
By the Fredholm theorem the spectrum
of this operator is discrete. Therefore, by the minimax principle, we have
$(\mathcal{A}_-v,v)_{L^2(\mathbb T^d)}\geqslant C_2\|v\|^2_{L^2(\mathbb T^d)} $, $C_2>0$,  for any $v$ whose average over $\mathbb T^d$
is equal to zero. Here $C_2$ is the distance from  zero to the rest of the spectrum of $\mathcal{A}_-$. Thus we have
$$
\begin{array}{c}
\displaystyle
\int_{\mathbb T^{d+1}}\int_{\mathbb R^d}a(\xi-\eta)\mu(\xi,\eta,s)\big(
\mathop{\beta}\limits^{\scriptscriptstyle \circ}(\eta,s)-\mathop{\beta}\limits^{\scriptscriptstyle \circ}(\xi,s))\big)^2\,d\eta d\xi ds
\\[5.5mm]
\displaystyle
\geqslant
\int_{\mathbb T^{d+1}}\int_{\mathbb R^d}\mu^-a(\xi-\eta)
\big(\mathop{\beta}\limits^{\scriptscriptstyle \circ}(\eta,s)-\mathop{\beta}\limits^{\scriptscriptstyle \circ}(\xi,s))\big)^2\,d\eta d\xi ds
\geqslant C_2 \|\mathop{\beta}\limits^{\scriptscriptstyle \circ}\|^2_{L^2(\mathbb T^{d+1})}, \quad C_2>0.
\end{array}
$$
Combining  \eqref{start_ineq} with the last estimate  and using the Cauchy-Schwarz inequality we deduce the estimate $ \|\mathop{\beta}\limits^{\scriptscriptstyle \circ}\|_{L^2(\mathbb T^{d+1})}\leqslant
C_2^{-\frac12}\|\mathop{\theta}\limits^{\scriptscriptstyle \circ}\|_{L^2(\mathbb T^{d+1})}$.
Since $\| \beta \|^2_{L^2(\mathbb T^{d+1})} = \| \mathop{\beta}\limits^{\scriptscriptstyle \circ} \|^2_{L^2(\mathbb T^{d+1})} + \| \overline{\beta} \|^2_{L^2(\mathbb T^{d+1})}$, we conclude that
 the estimate $\|\beta\|_{L^2(\mathbb T^{d+1})}\leq C_3\|\theta\|_{L^2(\mathbb T^{d+1})}$ holds.
Finally upper bound \eqref{aux2} follows from the standard parabolic estimates. Indeed, multiplying  equation \eqref{aux1}
by $\beta$ and integrating the resulting equation over the set $\mathbb T^d\times(s_1,s_2)$, $0\leqslant s_1<s_2\leqslant 1$,
we obtain
$$
\frac 12\|\beta(\cdot,s_1)\|^2_{L^2(\mathbb T^d)}-\frac 12\|\beta(\cdot,s_2)\|^2_{L^2(\mathbb T^d)}\leqslant
\int_{s_1}^{s_2}\int_{\mathbb T^d}\beta(\xi,s)\theta(\xi,s)d\xi ds\leqslant \|\beta\|_{L^2(\mathbb T^{d+1})}
\|\theta\|_{L^2(\mathbb T^{d+1})}
$$
$$
\leqslant C_3\|\theta\|^2_{L^2(\mathbb T^{d+1})}.
$$
This yields  \eqref{aux2}.
\end{proof}

\bigskip
We now introduce a periodic in  $\xi$ and $s$ vector-function
$$
\chi=\big(\chi_1(\xi,s),\,\chi_2(\xi,s),\ldots,\chi_d(\xi,s)\big), \qquad
 \chi\in L^\infty(0,1;(L^2(\mathbb T^d))^d),
 $$
 whose components are defined as periodic solutions of the following equations:
 \begin{equation}\label{aux3}
  \partial_s \chi_j(\xi,s)-\int_{\mathbb R^d}a(\xi-\eta)\mu(\xi,\eta,s)(\chi_j(\eta,s)-\chi_j(\xi,s))d\eta= -
 \int_{\mathbb R^d}a(\xi-\eta)\mu(\xi,\eta,s)(\eta_j-\xi_j)d\eta,
\end{equation}
$j=1,\ldots, d$.
 For brevity we denote the vector-function on the right-hand side of this equation by $q(\xi,s)=(q_1(\xi,s),\ldots,q_d(\xi,s))$.
One can easily check that under our standing assumptions this function is well-defined, periodic in  $\xi$ and $s$, and
belongs to the space $L^\infty(\mathbb T^{d+1})$.  Moreover, taking in account the symmetry conditions, we conclude that
 $$
 \int_{\mathbb T^d} q(\xi,s)d\xi=0
 $$
 for all $s$. Therefore, due to Proposition  \ref{p_auxi}, equation   \eqref{aux3} has a periodic solution. In order to fix the choice of
 an additive constant we impose a normalization consition
 $$
 \int_{\mathbb T^{d+1}} \chi(\xi,s)d\xi ds=0.
 $$
 This vector-function $\chi$ is called a corrector.

 We also define an effective matrix 
$a^{\mathrm{eff}}=\{a^{\mathrm{eff}}_{ij}\}_{i,j=1}^d$
by the formula
\begin{equation}\label{a_efff}
a^{\mathrm{eff}}_{ij}=\frac12\big(\widehat{a}_{ij}+\widehat{a}_{ji}\big)
\end{equation}  
with
 \begin{equation}\label{hat_eff}
 \widehat{a}_{ij}=\int_0^1\int_{\mathbb T^d}\int_{\mathbb R^{d}}a(\xi-\eta)\mu(\xi,\eta,s)\big(\frac12(\eta_i-\xi_i)(\eta_j-\xi_j)
 -(\eta_i-\xi_i)\chi_j(\eta,s)\big)d\xi d\eta ds.
 \end{equation}
  \begin{lemma}\label{l_positiv_def}
The matrix $a^{\mathrm{eff}}$ is positive definite.
 \end{lemma}

\begin{proof}
For an arbitrary vector $\zeta\in\mathbb R^d$  denote $\chi_\zeta(\eta,s)=\chi(\eta,s)\cdot\zeta$,
where the symbol $'\cdot'$ stands for the scalar product in $\mathbb R^d$.
Then by \eqref{hat_eff} we have
$$
 \widehat{a}\zeta\cdot\zeta=\int_0^1\int_{\mathbb T^d}\int_{\mathbb R^{d}}a(\xi-\eta)\mu(\xi,\eta,s)\big(\frac12(\eta_i-\xi_i)(\eta_j-\xi_j)
 -(\eta_i-\xi_i)\chi_j(\eta,s)\big)\zeta_i\zeta_j\,d\xi d\eta ds=
$$
$$
=\int_0^1\int_{\mathbb T^d}\int_{\mathbb R^{d}}a(\xi-\eta)\mu(\xi,\eta,s)\big(\frac12(\eta\cdot\zeta-\xi\cdot\zeta)^2
 -(\eta\cdot\zeta-\xi\cdot\zeta)\chi_\zeta(\eta,s)\big)\,d\xi d\eta ds;
$$
here and in what follows we assume a summation over repeating indices.
According to \eqref{aux3} the function $\chi_\zeta$ is a solution to the equation
 \begin{equation}\label{aux3_zeta}
  \partial_s \chi_\zeta(\xi,s)-\int_{\mathbb R^d}a(\xi-\eta)\mu(\xi,\eta,s)(\chi_\zeta(\eta,s)-\chi_\zeta(\xi,s))d\eta= -
 \int_{\mathbb R^d}a(\xi-\eta)\mu(\xi,\eta,s)(\eta\cdot\zeta-\xi\cdot\zeta)d\eta.
\end{equation}
Multiplying this equation by $\chi_\zeta(\xi,s)$
and integrating the resulting relation in variables $\xi$ and $s$ over  $\mathbb T^{d+1}$ yields
$$
\int_0^1\int_{\mathbb T^d}\int_{\mathbb R^{d}}a(\xi-\eta)\mu(\xi,\eta,s)\big(\frac12(\chi_\zeta(\eta,s)-\chi_\zeta(\xi,s))^2
 +(\eta\cdot\zeta-\xi\cdot\zeta)\chi_\zeta(\xi,s)\big)\,d\xi d\eta ds=0.
$$
Considering this relation we obtain
$$
\int_0^1\int_{\mathbb T^d}\int_{\mathbb R^{d}}a(\xi-\eta)\mu(\xi,\eta,s)\big((\eta\cdot\zeta-\xi\cdot\zeta)-(\chi_\zeta(\eta,s)-\chi_\zeta(\xi,s))\big)^2\,d\xi d\eta ds=
$$
$$
=\int_0^1\int_{\mathbb T^d}\int_{\mathbb R^{d}}a(\xi-\eta)\mu(\xi,\eta,s)\big((\eta\cdot\zeta-\xi\cdot\zeta)^2-2(\eta\cdot\zeta-\xi\cdot\zeta)\chi_\zeta(\eta,s))\big)\,d\xi d\eta ds+
$$
$$
+\int_0^1\int_{\mathbb T^d}\int_{\mathbb R^{d}}a(\xi-\eta)\mu(\xi,\eta,s)\big((\chi_\zeta(\eta,s)-\chi_\zeta(\xi,s))^2+2(\eta\cdot\zeta-\xi\cdot\zeta)\chi_\zeta(\xi,s))\big)\,d\xi  d\eta ds
$$
$$
=2 a^{\rm eff}\zeta\cdot\zeta.
$$
This implies the required positive definiteness of the matrix $a^{\rm eff}$.
\end{proof}

\section{Main result}\label{s_mainres}

We pass to the main result of this work.
 \begin{thm}\label{t_main1}
 Let conditions $\bf C_1$--$\bf C_5$ be fulfilled. Then for any initial condition $u_0\in L^2(\mathbb R^d)$ a solution
 $u^\eps$ of problem  \eqref{ori_eqn} converges as $\eps\to0$ in the space  $L^\infty(0,T;L^2(\mathbb R^d))$
 to a solution $u^0$ of the following homogenized problem;
 \begin{equation}\label{hom_eqn}
\begin{array}{c}
\displaystyle
\partial_t u^0(x,t)=\mathrm{div}\big(a^{\mathrm{eff}}\nabla u^0(x,t)\big)\qquad \hbox{in }\mathbb R^d\times[0,T],\\[2mm]
u^0(x,0)=u_0(x);
\end{array}
\end{equation}here $a^{\mathrm{eff}}$ is the effective matrix defined in \eqref{a_efff} and \eqref{hat_eff}.
 \end{thm}
 \begin{proof}
 We use the corrector approach and assume first that the function $u^0$ is smooth and that this function and its partial derivatives of any order decay at infinity. We introduce the following ansatz:
\begin{equation}\label{ansatz}
w^\eps(x,t)=u^0(x,t)+\eps\chi\Big(\frac x\eps,\frac t{\eps^2}\Big)\cdot\nabla u^0(x,t)+
\eps^2\varkappa\Big(\frac x\eps,\frac t{\eps^2}\Big)\cdot\nabla\nabla u^0(x,t);
\end{equation}
here and in what follows the notation $\nabla\nabla u^0$ stands for the matrix of partial derivatives of $u^0$ with respect
to the spatial variables: $\nabla\nabla u^0(x,t)=  \big\{\partial_{x_i}\partial_{x_j}   u^0(x,t)\big\}_{i,j=1}^d$.
Our goal is to choose a periodic vector-function $\chi$, a periodic matrix-function  $\varkappa$ and a matrix $a^{\mathrm{eff}}$
in such a way that the norm of the difference $(w^\eps-u^\eps)$ in the space $L^\infty(0,T;L^2(\mathbb R^d))$
tends to zero as $\eps\to0$. Recall that the vector-function $\chi$ and the matrix $a^{\mathrm{eff}}$ were already introduced in the
previous section.

Making a change of variables $(x,y)\to \big(x, \eps^{-1}(x-y)\big)$ and letting $z=\eps^{-1}(x-y)$ we can rewrite
formula  \eqref{ori_oper} as follows:
\begin{equation}\label{ori_oper_updated}
\big(\mathcal{A}^\eps(t) u\big)(x,t)=\frac1{\eps^{2}}\int_{\mathbb R^d}a\big(z\big)
\mu\Big(\frac{x}\eps,\frac x\eps-z, \
\frac t{\eps^2}\Big)\big(u(x-\eps z,t)-u(x,t)\big)dz,
\end{equation}
Expanding $u^0(x-\eps z,t)$ and $\nabla u^0(x-\eps z,t)$ into a Taylor series about $x$ we obtain
$$
u^0(x-\eps z,t)=u^0(x,t)-\eps z\cdot\nabla u^0(x,t)+\eps^2\int_0^1z_iz_j\partial_{x_i}\partial_{x_j}u^0(x-\eps\tau z,t)(1-\tau)d\tau,
$$
$$
\nabla u^0(x-\eps z,t)=\nabla u^0(x,t)-\eps\int_0^1z_j\partial_{x_j}\nabla u^0(x-\eps\tau z,t)d\tau.
$$
Combining the last two formulae with \eqref{ansatz} we derive the following expression for $w^\eps(x-\eps z, t)$:
$$
\begin{array}{ccc}
\displaystyle
  w^\eps(x-\eps z,t)= u^0(x,t)-\eps z\cdot\nabla u^0(x,t)+\eps^2\int_0^1z_iz_j\partial_{x_i}\partial_{x_j}u^0(x-\eps\tau z,t)(1-\tau)d\tau\\
  \displaystyle
  +\eps \chi\big(\frac x\eps - z,\frac t{\eps^2})\Big(\nabla u^0(x,t)-\eps\int_0^1z_j\partial_{x_j}\nabla u^0(x-\eps\tau z,t)d\tau\Big)
  +\eps^2\varkappa\big(\frac x\eps - z,\frac t{\eps^2})\nabla\nabla u^0(x,t).
\end{array}
$$
Our next goal is to compute $\partial_t w^\eps(x,t)- \mathcal{A}^\eps(t) w^\eps(x,t)$. Considering  \eqref{ori_oper_updated} and the last two relations,  after straightforward transformations we have
\begin{equation}\label{like_powers}
\begin{array}{ccc}
 \displaystyle
\partial_t w^\eps(x,t)- \mathcal{A}^\eps (t) w^\eps(x,t)=   \\[2mm]
 \displaystyle
  = \frac1\eps\nabla u^0(x,t)\cdot\Big\{\partial_s \chi\Big(\frac x\eps,s\Big)-\int_{\mathbb R^d}a(z)\mu\Big(\frac x\eps,\frac x\eps-z,s\Big)
  \Big[\chi\Big(\frac x\eps-z,s\Big)-\chi\Big(\frac x\eps,s\Big)-z\Big]dz\Big\}_{s=\frac t{\eps^2}}
   \\[4mm]\displaystyle
  +\partial_tu^0(x,t)+\nabla\nabla u^0(x,t)\cdot\Big\{\partial_s \varkappa\Big(\frac x\eps,s\Big)
  -\int_{\mathbb R^d}a(z)\mu\Big(\frac x\eps,\frac x\eps-z,s\Big)
  \Big[\varkappa\Big(\frac x\eps-z,s\Big)-\varkappa\Big(\frac x\eps,s\Big)\Big]dz\\[4mm]\displaystyle
  -\int_{\mathbb R^d}a(z)\mu\Big(\frac x\eps,\frac x\eps-z,s\Big) \Big[\frac12 z\otimes z-
  z\otimes\chi\Big(\frac x\eps-z,s\Big) \Big]dz \Big\}_{s=\frac t{\eps^2}} +R^\eps(x,t)
\end{array}
\end{equation}
with
$$
\begin{array}{cc}
\displaystyle
R^\eps(x,t)=\eps\chi\Big(\frac x\eps,\frac t{\eps^2}\Big)\partial_t\nabla u^0(x,t)+
\eps^2\varkappa\Big(\frac x\eps,\frac t{\eps^2}\Big)\partial_t\nabla\nabla u^0(x,t)\\[4mm]
\displaystyle
-\int_{\mathbb R^d}a(z)\mu\Big(\frac x\eps,\frac x\eps-z,\frac t{\eps^2}\Big)
z_iz_j\int_0^1\big(\partial_{x_i}\partial_{x_j}u^0(x-\eps\tau z,t)-\partial_{x_i}\partial_{x_j}u^0(x,t)\big)(1-\tau)d\tau dz\\[4mm]
\displaystyle
-\int_{\mathbb R^d}a(z)\mu\Big(\frac x\eps,\frac x\eps-z,\frac t{\eps^2}\Big)\chi_i\Big(\frac x\eps-z, \frac t{\eps^2}\Big)z_j
\int_0^1\big(\partial_{x_i}\partial_{x_j}u^0(x-\eps\tau z,t)-\partial_{x_i}\partial_{x_j}u^0(x,t)\big)d\tau dz;
\end{array}
$$
here the symbol $\otimes$ stands for the tensor product.
According to \cite[Proposition 5]{PiaZhi17}, for any $\varkappa\in L^2(\mathbb T^{d+1})$ and for any smooth $u^0$ whose derivatives of any order decay at infinity faster that any negative power of $|x|$, we have
$\|R^\eps\|_{L^2(\mathbb R^d\times[0,T])} \to 0$, as $\eps\to0$.
Due to equation \eqref{aux3} the first expression in figure brackets on the right-hand side of  \eqref{like_powers} vanishes.
Letting $u^0$ be a solution of equation  \eqref{hom_eqn} and recalling \eqref{hat_eff} we have $\partial_t u^0(x,t)=\widehat{a}\nabla\nabla u^0(x,t)$, and
\eqref{like_powers} can be rearranged as
\begin{equation}\label{like_powers_upd}
\begin{array}{ccc}
 \displaystyle
\partial_t w^\eps(x,t)- \mathcal{A}^\eps (t) w^\eps(x,t)=   \\[2mm]
 \displaystyle
  =  \nabla\nabla u^0(x,t)\cdot\bigg\{\partial_s \varkappa\Big(\frac x\eps,s\Big)
  -\int_{\mathbb R^d}a(z)\mu\Big(\frac x\eps,\frac x\eps-z,s\Big)
  \Big[\varkappa\Big(\frac x\eps-z,s\Big)-\varkappa\Big(\frac x\eps,s\Big)\Big]dz\\[4mm]\displaystyle
  -\int_{\mathbb R^d}a(z)\mu\Big(\frac x\eps,\frac x\eps-z,s\Big) \Big[\frac12 z\otimes z-
  z\otimes\chi\Big(\frac x\eps-z,s\Big) \Big]dz - \widehat a \bigg\}_{s=\frac t{\eps^2}} +R^\eps(x,t)
\end{array}
\end{equation}
In view of \eqref{hat_eff} the compatibility condition in the equation
$$
\partial_s\varkappa(\xi,s)-\int_{\mathbb R^d}a(z)\mu(\xi,\xi-z,s)\Big[\varkappa(\xi,s)-\varkappa(\xi-z,s)+\frac12
z\otimes z-z\otimes\chi(\xi-z,s)-\widehat{a}\Big]dz = 0
$$
is fulfilled, and thus this equation has a solution  $\varkappa\in (L^\infty(0,1;L^2(\mathbb T^d)))^{d^2}$. Inserting this solution to
\eqref{like_powers_upd} yields
$$
\partial_t w^\eps(x,t)- \mathcal{A}^\eps (t) w^\eps(x,t)= R^\eps(x,t).
$$
Then the difference $(w^\eps-u^\eps)$ satisfies the relations
$$
\begin{array}{c}
\partial_t (w^\eps-u^\eps)- \mathcal{A}^\eps (t) (w^\eps - u^\eps)= R^\eps, \\[2mm]
\displaystyle
 w^\eps(x,0)-u^\eps(x,0)=\eps\chi\Big(\frac x\eps,0\Big) \cdot \nabla u_0(x)+\eps\varkappa\Big(\frac x\eps,0\Big)\cdot \nabla\nabla u_0(x).
\end{array}
$$
If $u_0\in C_0^\infty(\mathbb R^d)$, then $\|w^\eps(\cdot,0)-u^\eps(\cdot,0)\|_{L^2(\mathbb R^d)}\to 0$ as $\eps\to0$,
and $\|R^\eps\|_{L^\infty(0,T;L^2(\mathbb R^d))}\to 0$.

In order to complete the proof of Theorem \ref{t_main1} we need a priori estimates. Consider in $\mathbb R^d\times(0,T]$ a  Cauchy problem
\begin{equation}\label{arb_cauchy}
  \partial_t v-\mathcal{A}^\eps(t)v=f(x,t),\quad v(x,0)=v_0.
\end{equation}
\begin{prop}\label{p_a_priori}
  For a solution $v^\eps$ of problem \eqref{arb_cauchy} the following estimate holds
\begin{equation}\label{prop2}
  \|v^\eps\|_{L^\infty(0,T;L^2(\mathbb R^d))}\leq C_4 \|f\|_{L^2(\mathbb R^d\times(0,T))}+\|v_0\|_{L^2(\mathbb R^d)}.
\end{equation}
  with a constant $C_4>0$ that does not depend on $\eps$.
\end{prop}
\begin{proof}
The proof of this statement is quite standard and follows the  line of the proof of Proposition 6.1 in \cite{PZ19}. For the reader convenience here we provide a sketch of the proof. Since problem \eqref{arb_cauchy} is linear its solution can be represented as the sum
$$
v^\eps(x,t) = v_1(x,t) + v_2(x,t),
$$
where $v_1$ and $v_2$ are solutions of  the problems
\begin{equation}\label{Prop2-1}
  \partial_t v_1 - \mathcal{A}^\eps(t)v_1 = 0,\quad v_1(x,0)=v_0,
\end{equation}
and 
\begin{equation}\label{Prop2-2}
  \partial_t v_2 - \mathcal{A}^\eps(t)v_2 = f(x,t),\quad v_2(x,0)=0,
\end{equation}
respectively.
Multiplying equation \eqref{Prop2-1} by $v_1(x,s)$ and integrating the resulting relation over $\mathbb{R}^d\times(0,t)$ we obtain
\begin{equation}\label{E1}
\| v_1(\cdot, t)\|_{L^2(\mathbb{R}^d)} \ \le \ \| v_0(\cdot)\|_{L^2(\mathbb{R}^d)} \quad \mbox{for any } \; t \in [0,T].
\end{equation}
Making similar computations in problem \eqref{Prop2-2} yields  
$$
\| v_2(\cdot, t)\|^2_{L^2(\mathbb{R}^d)} \ \le \ 2\, \Big| \int\limits_0^t \int\limits_{\mathbb{R}^d} f(x,s) v_2(x,s) dx ds \Big| \ \le \ \|f\|^2_{L^2( \mathbb{R}^d \times (0,T))} + \int\limits_0^t \| v_2(\cdot, s)\|^2_{L^2(\mathbb{R}^d)} ds.
$$
By the Gronwall inequality we have
\begin{equation}\label{E2}
\| v_2(\cdot, t)\|^2_{L^2(\mathbb{R}^d)} \ \le \ \|f\|^2_{ L^2(\mathbb{R}^d  \times (0,T))} \, e^T \quad \mbox{for any } \; t \in [0,T].
\end{equation}
Finally, \eqref{E1} and \eqref{E2} imply \eqref{prop2}.
\end{proof}
By this Proposition taking into account the relations $\|w^\eps(\cdot,0)-u^\eps(\cdot,0)\|_{L^2(\mathbb R^d)}\to 0$ and
$\|R^\eps\|_{L^\infty(0,T;l^2(\mathbb R^d))}\to 0$ we have $\|u^\eps-w^\eps\|_{L^\infty(0,T;L^2(\mathbb R^d))}\to 0$, as $\eps\to0$.
Considering the relation  $\|u^0-w^\eps\|_{L^\infty(0,T;L^2(\mathbb R^d))}\to 0$ that follows from the structure of
ansatz  \eqref{ansatz} we finally conclude that
$$
\lim\limits_{\eps\to0}\|u^\eps-u^0\|_{L^\infty(0,T;L^2(\mathbb R^d))}= 0.
$$
Approximating any $L^2$ initial function by a sequence of smooth functions and taking into account the a priori estimates proved in Proposition \ref{p_a_priori} and similar estimates for the limit problem we derive the desired statement of Theorem \ref{t_main1}.
 \end{proof}

\section{Acknowledgment}
This work was partially supported by the UiT Aurora project MASCOT and  Pure Mathematics in Norway grant.

\end{document}